\documentclass{amsproc}
\usepackage{amsmath, amssymb, amsthm, amsfonts}
\newtheorem{theorem}{Theorem}[section]
\newtheorem{lemma}[theorem]{Lemma}
\newtheorem{definition}[theorem]{Definition}

\newcommand{\rr}{\mathbb R}

\newcommand{\lip}{\mbox{lip }}
\newcommand{\Lip}{\mbox{Lip }}
\newcommand{\zz}{\mathbb{Z}}
\newcommand{\nn}{\mathbb{N}}

\newcommand{\bde}{\begin{definition}}
\newcommand{\ede}{\end{definition}}
\newcommand{\beq}{\begin{equation}}
\newcommand{\eeq}{\end{equation}}

\newcommand{\be}{\begin{enumerate}}

\newcommand{\ee}{\end{enumerate}}

\begin{document}
\title{Sets where $\Lip f$ is infinite and $\lip f$ is finite.}

\author{ Bruce Hanson, Department of Mathematics,\\ Statistics and Computer Science,\\ St.\ Olaf College,
Northfield, Minnesota 55057, USA\\
{email:} \texttt{hansonb@stolaf.edu}}


\Large

\begin{abstract}

We characterize the sets $E\subset \rr$ such that there exists a continuous function $f:\rr \to \rr$ with $\lip f$ finite everywhere and $\Lip f$ infinite precisely on $E$. 

\end{abstract}

\maketitle

\section{Introduction}Throughout this note we assume that $f:\rr \to \rr$ is continuous.
Then the so-called ``big Lip'' and ``little lip'' functions are defined as follows:

$$
\Lip f(x)=\limsup_{r\rightarrow 0^+}M_f(x,r) \mbox{ and } \lip f(x)=\liminf_{r\rightarrow 0^+}M_f(x,r), $$
where $$M_f(x,r)=\sup\{\frac{|f(x)-f(y)|}r \colon |x-y| \le r\}.$$ 
We also define $L_f^\infty=\{x\in \rr : \, \Lip f(x)=\infty\}$ and $l_f^\infty=\{x\in \rr: \, \lip f(x)=\infty\}$.

Results about the big Lip function date back to the early days of the 20th century.  For example, the proof of the well-known Rademacher-Stepanov Theorem (see \cite{eg}), which states that $f $ is differentiable almost everywhere on the complement of $L_f^\infty$, is almost 100 years old.  (See \cite{m} for an elegant proof of this result.)  On the other hand, the little lip function is a more recent phenomenon.  As far as I know, the first reference to the little lip function occurs in a 1999 paper by Cheeger, \cite{ch}.   Another early reference occurs in (\cite{bc}), where the authors 
show that the sets $L_f^\infty$ and $l_f^\infty$ can differ greatly. There they construct a continuous function $f:\rr\to\rr$ such that $L_f^\infty=\rr$, but $|l_f^\infty|=0$ (here $|S|$ denotes the Lebesgue measure of $S$).  In fact, in their example $\lip f(x)=0 $ on
$\rr\backslash E$, where $|E|=0$. In (\cite{h1}) the author of this note shows that it is possible to make the exceptional set $E$ have Hausdorff dimension 0.

On the other hand, it is impossible to construct a function $f:\rr\to\rr$ such that $L_f^\infty=\rr$ and $\l_f^\infty=\emptyset$.  Balogh and Cs\"ornyei showed (\cite{bc}) that if $l_f^\infty=\emptyset$, then every non-degenerate interval $I \subset \rr$ contains a set of positive measure on which $f$ is differentiable.  It follows that in this case $|L_f^\infty \cap I|<|I|$ for each non-degenerate interval $I$.  This motivates the following definition:

\begin{definition}
\label{trim}
A subset $E$ of $\rr$ is \underline{trim} if $|E \cap (a,b)| < b-a $ for all open intervals $(a,b)$.
\end{definition}

It is straightforward to show that $L_f^\infty$ is a $G_\delta$ set for any continuous function $f$.  Thus, if
$l_f^\infty=\emptyset$, then $L_f^\infty$ is a trim $G_\delta$ set.  In this note we show that this gives a characterization of trim $G_\delta$ sets:

\begin{theorem}\label{trim g delta}
Suppose that $f:\rr\to\rr$ is continuous and $l_f^\infty=\emptyset$.  Then $L_f^\infty$ is a trim $G_\delta$ set. Conversely,
if $E\subset \rr$ is a trim $G_\delta$ set, then there exists a continuous function $f:\rr\to\rr$ such that $L_f^\infty=E$ and $l_f^\infty=\emptyset$.  Moreover, $f$ may be constructed so that $\lip f(x)=0$ for all $x \in E$, as well.
\end{theorem}

Now suppose that $f$ is a monotone function.  Then $f$ is differentiable almost everywhere so $|L_f^\infty|=0$.  In this case, we have the following result:

\begin{theorem}\label{monotone} Suppose that $f:\rr\to\rr$ is continuous and monotonic.  Then $L_f^\infty$ is a $G_\delta$ set of measure zero.  Conversely, if
$E\subset\rr$ is a $G_\delta$ set with measure zero, then there exists a continuous, monotonic function $f:\rr\to\rr$ such that $L_f^\infty=E$ and $l_f^\infty=\emptyset$.  Moreover, $f$ may be constructed so that $\lip f(x)=0$ for all $x \in E$.
\end{theorem}

The layout of this paper is the following: In section 2 we present our notation and state a few useful lemmas.  Sections 3 and 4 contain the proofs of Theorems \ref{trim g delta} and \ref{monotone}, respectively.  In the final section we offer a few open problems that are connected to Theorems \ref{trim g delta} and \ref{monotone}. 

Other interesting recent results concerning the big and little lip functions can be found in \cite{bhmv1, bhmv2, bhmv3, h2}

\section{Definitions and Lemmas}

\begin{definition}\label{one side}
Given a continuous function $f:\rr\to\rr$, we define
$$\Lip^+ f(x)=\limsup_{r\to0^+}\sup_{0\le y-x \le r} \frac{|f(x)-f(y)|}r,$$
and
$$\Lip^- f(x)=\limsup_{r\to0^+}\sup_{0\le x-y \le r} \frac{|f(x)-f(y)|}r.$$
\end{definition}

Note that $\Lip f(x)=\max\{\Lip^+f(x),\Lip^-f(x)\}$. We leave the proof of the following simple lemma up to the reader.

\begin{lemma}\label{lip squeeze}
Suppose that $g$ and $h$ are Lipschitz on $[a,b]$ and that $g(x)\le f(x) \le h(x)$ for all $x \in [a,b]$ and
$A=\{x \in [a,b]: \, g(x)=h(x)\}.$  Then
\beq\label{lip on A}
\Lip f(x)<\infty \mbox{ for all } x \in A \cap (a,b)
\eeq
and $\Lip^+f(a) <\infty$ if $a \in A$ and $\Lip^-f(b)<\infty$ if $b \in A$.
\end{lemma}

\begin{definition}\label{phi}
Given a closed interval $I=[a,b]$ and $n \in \nn$, we define
\beq\label{phi def}
\Phi_{n,I}(x))= \left\{ \begin{array}{cc}
            \frac1{2^n}\min\{x-a,b-x\} & \mbox{if } x\in [a,b] \\
            0  & \mbox{if } x\notin [a,b]
             \end{array}
       \right.
\eeq
\end{definition}

\begin{definition}\label{n close}
Let $Z=\{z_j:\, j\in \zz\}\subset (a,b)$.  We say that $Z$ is \underline{$n$-close on $(a,b)$} if
\beq\label{zj monotone}
a<z_j<z_{j+1}<b \mbox{ for all }j \in \zz,
\eeq
\beq\label{zj lim a}
\lim_{j\to \infty}z_j=b \mbox{ and } \lim_{j\to -\infty}z_j=a,
\eeq
and
\beq\label{zj geom}
z_{j+1}-z_j < \frac1{4^n}\min\{z_j-a,b-z_{j+1}\} \mbox{ for all }j \in \zz.
\eeq
\end{definition}

\begin{definition}\label{n close U}
Given an open set $U$ and a countable set
$Z\subset U$, we say that $Z$ is \underline{$n$-close on $U$} if $Z \cap (a,b)$ is $n$-close on $(a,b)$ for each
component $(a,b)$ of $U$.
\end{definition}

\begin{definition}\label{zig zag}
Suppose that $Z=\{z_j : \, j\in \zz\}$ is $n$-close on $[a,b]$.  Then we say that $f:[a,b]\to \rr$ is \underline{zig-zag of order $n$} with respect to $Z$ on $[a,b]$ if $f(a)=f(b)$ and

\beq\label{zig zag def}
f(z_j)= \left\{ \begin{array}{cc}
            f(a) & \mbox{if } j \mbox{ is even} \\
            f(a)+\Phi_{n,[a,b]}(z_j) & \mbox{if } j \mbox{ is odd}
             \end{array}
       \right.
\eeq
and
\beq\label{linear on}
f \mbox{ is linear on }[z_j,z_{j+1}] \mbox{ for all }j \in \zz.
\eeq
\end{definition}

\begin{definition}\label{trim g}
Suppose that $f$ is linear on $[a,b]$ and $U$ is an open set such that $|U\cap [a,b]| < b-a$ and $\{a,b\}\cap U=\emptyset.$  Then we define $g=g(f,[a,b],U)$ on $[a,b]$ as follows:
$$ g(x)=f(a)+\frac{|[a,b]\backslash U)\cap [a,x]|}{|[a,b]\backslash U|}(f(b)-f(a)).$$
\end{definition}

The proof of the following lemma is straightforward and left to the reader.

\begin{lemma}\label{linear lip}
Suppose that $f$ is linear on $[a,b]$ and $U$ is an open set such that $|U\cap [a,b]|<b-a$ and $\{a,b\}\cap U=\emptyset$. Let $g=g(f,[a,b],U)$.

Then
\beq\label{lip on ab}
g \mbox{ is Lipschitz on }[a,b],
\eeq

\beq\label{g const}
g \mbox{ is constant on each interval }(r,s) \mbox{ contained in }[a,b]\cap U,
\eeq

\beq\label{g equal f}
g(a)=f(a) \mbox{ and }g(b)=f(b),
\eeq
and
\beq\label{fg close}
||g-f||_\infty \le |f(b)-f(a)|.
\eeq
\end{lemma}

\begin{definition}\label{trim in ab}
Given an open set $U$ and a closed interval $[a,b]$, we say that $U$ is\underline{ trim in $[a,b]$} if $|U\cap [a,b]|<b-a$ and $\{a,b\}\cap U =\emptyset$. Furthermore, if $\mathcal{F}=\{[a_n,b_n]:\, n \in \nn\}$ is a collection of pairwise non-overlapping closed interevals, we say that $U$ is\underline{ trim on $\mathcal{F}$} if $U$ is trim in $[a_n,b_n]$ for all $n \in \nn$.
\end{definition}

\begin{lemma}\label{n close zig zag}
Let $n \in \nn$.  Suppose that $Z=\{z_i: \, i\in \zz\}$ is $n$-close on $[a,b]$ and $F$ is zig-zag of order $n$ with respect to $Z$ on $[a,b]$.  Then
\beq\label{F steep}
|F(z_i)-F(z_{i+1})|\ge 2^n |z_i-z_{i+1}| \mbox{ for all }i \in \nn
\eeq

Furthermore, let $\mathcal{F}=\{[z_i,z_{i+1}]:\, i \in \zz\}$ and assume
$U=\sqcup_{i\in J} (a_i,b_i)=\sqcup_{i\in J} I_i\subset(a,b)$ is trim on $\mathcal{F}$ and $G:[a,b] \to \rr$ satisfies
\beq\label{G and F}
G(a)=G(b)=F(a)
\eeq
and
\beq\label{G restrict F}
G\vert_{[z_i,z_{i+1}]}=g(F,[z_i,z_{i+1}],U) \mbox{ for all }i \in \zz.
\eeq
Define
$$H(x)=G(x)+\sum_{i\in J} \Phi_{n,[a_i,b_i]}(x) \mbox{ for all }x \in [a,b].$$
Then
\beq\label{H less G}
H(x)\le G(a)+2\Phi_{n,[a,b]}(x) \mbox{ for all }x \in [a,b].
\eeq
\end{lemma}

\begin{proof} Note that (\ref{F steep}) follows directly from the definition of $F$ and (\ref{zj geom}).  Let $x \in [a,b]$.  Since equality holds in (\ref{H less G}) when $x=a$ or $x=b$, we may assume that $x \in (a,b)$.  Choose $j\in \zz$ such that $x \in [z_j,z_{j+1}]$.  We only consider the case where $j$ is even and $z_j-a \le b-z_{j+1}$.  (The other cases are handled by a similar argument.)  Note that the last inequality implies that
\beq\label{phi ineq}
\Phi_{n,[a,b]}(z_j) \le \Phi_{n,[a,b]}(y) \mbox{ for all }y \in [z_j,z_{j+1}].
\eeq
Since $j$ is even, it follows from the definition of $G$ that
\beq
G(x) \le F(z_{j+1}) = G(a)+\Phi_{n,[a,b]}(z_{j+1}).
\eeq
Note also that 
$$\sum_{i\in J}\Phi_{n,[a_i,b_i]}(x) \le \Phi_{n,[z_j,z_{j+1}]}(x),$$
and so we get 
\beq \label{H less than G}
H(x) \le G(a)+\Phi_{n,[a,b]}(z_{j+1})+\Phi_{n,[z_j,z_{j+1}]}(x).
\eeq

Now, using the fact that $Z$ is $n$-close on $[a,b]$, the definition of $\Phi_{n,[a,b]}$ and the assumption that
$z_j-a \le b-z_{j+1}$, we get
\beq\label{1 over 4n}
x-z_j\le z_{j+1}-z_j \le \frac1{4^n}(z_j-a)= \frac1{2^n}\Phi_{n,[a,b]}(z_j).
\eeq

From (\ref{1 over 4n}) it follows that
\beq\label{phi 1 over 2n}
\Phi_{n,[a,b]}(z_{j+1})\le \Phi_{n,[a,b]}(z_j)+\frac1{2^n}(z_{j+1}-z_j)\le (1+\frac1{4^n})\Phi_{n,[a,b]}(z_j)
\eeq
and then using (\ref{1 over 4n}) and (\ref{phi ineq}), we have
\beq\label{phi 1 over}
\Phi_{n,[z_j,z_{j+1}]}(x)\le \frac1{2^{n}}(x-z_j)\le \frac1{4^n}\Phi_{n,[a,b]}(z_j)\le\frac1{4^n}\Phi_{n,[a,b]}(x).
\eeq
Now (\ref{H less G}) follows from (\ref{H less than G}), (\ref{phi 1 over 2n}), and (\ref{phi 1 over}) and the proof of the lemma is complete.

\end{proof}

\section{Proof of Theorem \ref{trim g delta}}

\begin{proof}

Let $E$ be a trim $G_\delta$ set.  It suffices to show that for every $a<b$ with $a,b  \in \rr\backslash E$ we can construct $f:[a,b]\to \rr$ with $\lip f=0$ on $E \cap (a,b)$ and such that $L_f^\infty \cap(a,b) =E \cap (a,b)$, $\Lip^+f(a)<\infty$ and $\Lip^-f(b)<\infty$. We assume without loss of generality that $(a,b)=(0,1)$ and $E\subset (0,1)$.  So our goal is to construct a function $f: [0,1]\to \rr$ such that
\beq\label{Lf E}
L_f^\infty=E
\eeq
\beq\label{Lip f(0)}
\Lip^+f(0) < \infty \mbox{ and } \Lip^-f(1)< \infty
\eeq
and
\beq\label{lip f 0}
\lip f=0 \mbox{ on }E.
\eeq

Let $F=(0,1)\backslash E$.  Because $E$ is trim, it follows that $F$ is dense in $(0,1)$ and that given any collection $\mathcal{F}$ of pairwise non-overlapping closed intervals whose endpoints are in $F$, we can choose an open set $U$ such that $U$ is trim on $\mathcal{F}$ and $E\subset U$.

We first make use of the trimness of $E$ and the denseness of $F$ to construct a sequence of open sets $\{U_n\}$, countable sets $\{Z_n\}$, and collections of closed intervals $\{\mathcal{F}_n\}$ which will be used to construct $f$.  Let $\{V_n\}_{n=1}^\infty$ be a collection of open sets such that $V_{n+1}\subset V_n$ for all $n \in \nn$ and $E=\cap_{n=1}^\infty V_n$ and define $U_1=(0,1)$.  Next choose $Z_1=\{z_j\}_{j\in\zz}$ such that 
$Z_1 \subset F$ and $Z_1 $ is 1-close on $U_1$ and let $\mathcal{F}_1=\{[z_j,z_{j+1}]\}_{j \in \zz}$.  Now choose $U_2=\sqcup_{j\in J_2} (a_{2,j},b_{2,j})$ so that $U_2$ is trim on $\mathcal{F}_1$ and $E\subset U_2 \subset V_2$.  We proceed by choosing $Z_2=\cup_{j\in J_2} Z_{2,j}$ where 
$Z_{2,j}=\{z_{2,j,i} \,:\, i \in \zz\}$ is 2-close on $(a_{2,j},b_{2,j})$ for each $j \in J_2$.  We also require that $Z_2 \subset F$.  Then we define $\mathcal{F}_2=\{[z_{2,j,i},z_{2,j,i+1}] \,: \, j \in J_2, i \in \zz\}$. 

 Making use of the trimness of $E$ and the denseness of $F$ we proceed with an inductive argument and choose a sequence of open sets $\{U_n\}$, countable sets $\{Z_n\}$ and collections of closed intervals $\{\mathcal{F}_n\}$ satisfying the following: 

\beq\label{U inside V }
E \subset U_n \subset V_n \mbox{ for } n \in \nn. 
\eeq

\beq\label{Un an}
U_n=\sqcup_{j\in J_n} (a_{n,j},b_{n,j}) \mbox{ for }n \ge 2
\eeq


For each $n \ge 2 \ $ $Z_n=\cup_{j\in J_n} Z_{n,j}$, where
\beq \label{Zn n close}
Z_{n,j}=\{z_{n,j,i}:\, i\in \zz\} \mbox{ is $n$-close on } (a_{n,j},b_{n,j}) \mbox{ for }j \in J_n
\eeq

\beq\label{union Zn}
\cup_{n=1}^\infty Z_n \subset F
\eeq


\beq\label{Fn 2}
\mathcal{F}_n=\{[z_{n,j,i},z_{n,j,(i+1)}] : \, j\in J_n, \, i \in \zz\} \mbox{ for } n\ge2
\eeq

\beq\label{Un trim}
U_{n+1} \mbox{ is trim on } \mathcal{F}_n \mbox{ for each }n \in \nn.
\eeq

Note that for each $n \in \nn$ we have
\beq\label{Fn subset Un}
\cup_{[a,b]\in \mathcal{F}_n}[a,b]=U_n
\eeq
and
\beq\label{Zn subset U}
Z_n \subset U_n\backslash U_{n+1}.
\eeq

We now begin the construction of $f$.  Define $g_1=0$ on $[0,1]$ and $f_1$ to be zig-zag of order 1 with respect to $Z_1$ on $U_1=(0,1)$.  Proceeding inductively, for each $n \ge 2$ we define $g_n$ so that
\beq\label{Fn gn g}
\mbox{ for all } [a,b]\in \mathcal{F}_{n-1}, \ g_n\vert_{[a,b]}=g(f_{n-1},[a,b],U_n)
\eeq
and
\beq\label{gn fn-1}
g_n(x)=f_{n-1}(x) \mbox{ for all }x \in [0,1]\backslash U_{n-1}
\eeq
and we define $f_n$ so that for all $j \in \nn$
\beq\label{fn zig zag}
f_n \mbox{ is zig-zag of order $n$ with respect to }Z_{n,j} \mbox{ on }[a_{n,j},b_{n,j}]
\eeq
and
\beq\label{fn equal gn}
f_n(x)=g_n(x) \mbox{ for all } x \in [0,1]\backslash U_n.
\eeq
Now define
\beq\label{h1 phi}
h_1=2\Phi_{1,[0,1]}
\eeq
and for each $n \ge 2$ define 
\beq\label{hn def}
h_n=g_n+\sum_{j\in J_n} 2\Phi_{n,[a_{n,j},b_{n,j}]}.
\eeq

We note that from (\ref{Fn gn g})--(\ref{fn equal gn}), the definition of $\Phi_n$  and (\ref{fg close}) in Lemma \ref{linear lip} we have:

\beq\label{fn gn norm}
||f_n-g_n\vert\vert_\infty \le\frac1{2^n}
\eeq
and 
\beq\label{gn fn-1 norm}
||g_n-f_{n-1}||_\infty \le \frac1{2^{n-1}}.
\eeq.

Moreover, from (\ref{hn def}) and the definitions of $f_n$ and $g_n$ it follows that 
\beq\label{gfh}
g_n\le f_n \le h_n.
\eeq

It follows from Lemma \ref{linear lip} that for every $n \ge 2$ we have
\beq\label{gh Lip}
g_n \mbox{ and } h_n \mbox{ are Lipschitz on }[a,b] \mbox{ for every }[a,b]\in \mathcal{F}_{n-1}
\eeq
and from (\ref{Fn gn g}) and (\ref{gn fn-1}) we deduce that
\beq\label{gn fn-12}
g_n=f_{n-1} \mbox{ on }([0,1]\backslash U_{n-1})\cup Z_{n-1} \mbox{ for all }n\ge2.
\eeq

Using (\ref{gn fn-12}) and (\ref{fn equal gn}), we get
\beq\label{fn+1=fn}
f_{n}=f_{n-1} \mbox{ on }([0,1]\backslash U_{n-1})\cup Z_{n-1} \mbox{ for all }n \ge 2
\eeq
and therefore, taking (\ref{Zn subset U}) into account, we get
\beq\label{fk=fn}
\mbox{ for all }n\in \nn \mbox{ and for all } k \ge n \, f_k=f_n \mbox{ on }Z_n.
\eeq

It follows from (\ref{fn gn norm}) and (\ref{gn fn-1 norm}) that the sequence $\{f_n\}$ converges uniformly to a function $f$ on
$[0,1]$ and since each $f_n$ is clearly continuous on $[0,1]$, we conclude that $f$ is continuous on $[0,1]$ as well.  We need to show that $L_f^\infty=E$, $\Lip^+f(0)<\infty$, $\Lip^-f(1) < \infty$  and $\lip f=0$ on $E$.

We first note that (\ref{fn+1=fn}) implies that for every $n \in \nn$ we have $f=f_n$ on $[0,1]\backslash U_n$.  Next, using Lemma \ref{n close zig zag}, it is straightforward to show that $h_{n+1}\le h_n$ for $n \in \nn$ and from the construction of the $g_ns$ it is not hard to see that $g_n \le g_{n+1}$ for all $n \in \nn$.  Therefore, by (\ref{gfh}), we have
\beq\label{gn less f less hn}
g_n \le f \le h_n \mbox{ for all }n \in \nn.
\eeq
Using (\ref{gn less f less hn}) in the case $n=1$, we have that $0 \le f \le 2\Phi_{1,[0,1]}$ and it follows easily that
$\Lip^+f(0)\le 2$ and $\Lip^-f(1) \le 2$.

Fix $n \ge 2$ and let $[a,b]\in \mathcal{F}_{n-1}$.  Then from (\ref{gh Lip}), (\ref{gn less f less hn}) and Lemma \ref{lip squeeze} we conclude that $\Lip f$ is finite on $(a,b)\backslash U_n$ and $\Lip^+f(a) < \infty$ and $\Lip^-f(b)<\infty$.  Using (\ref{Fn subset Un}), we conclude that $\Lip f$ is finite on $U_{n-1}\backslash U_n$ and unfixing $n$ we get that $\Lip f$ is finite on
$\cup_{n=2}^\infty (U_{n-1}\backslash U_n)=(0,1)\backslash E$.

We next show that $\Lip f=\infty$ on $E$.  We begin by observing that from (\ref{fn+1=fn}) and (\ref{Zn subset U}) it follows that $f=f_n$ on $Z_n$ for all $n \in \nn$.  Let $x \in E$.  For each $n \in \nn$ choose $[a_n,b_n]\in \mathcal{F}_n$ such that $x \in [a_n,b_n]$ and let $r_n=b_n-a_n$.
Then using (\ref{F steep}) from Lemma \ref{n close zig zag}, we see that
$$ M_f(x,r_n)\ge \frac{|f(b_n)-f(a_n)|}{2r_n}=\frac{|f_n(b_n)-f_n(a_n)|}{2r_n}\ge 2^{n-1}.$$
Since $r_n \to 0$, we conclude that $\Lip f(x)=\infty$, as desired.

It remains to show that $\lip f(x)=0$ for all $x \in E$.  Let $x \in E$, so $x \in U_n$ for all $n \in \nn$.  For each $n\ge 2$ choose $j(n)$ so that $x \in (a_{n,j(n)},b_{n,j(n)})\subset U_n$.  For notational convenience we let $a_n=a_{n,j(n)}$ and $b_n=b_{n,j(n)}$ so we have $x \in (a_n,b_n) \subset U_n$ for all $n \ge 2$.  Fix $n \ge 2$ momentarily.  From (\ref{Fn gn g}) we see that $g_n$ is constant on $[a_n,b_n]$.  Then using (\ref{hn def}) and (\ref{gn less f less hn}), we deduce that
\beq\label{gfg}
g_n(a_n)\le f(y) \le g_n(a_n)+2\Phi_{n,[a_n,b_n]}(y) \mbox{ for all }y \in [a_n,b_n].
\eeq
Letting $s_n=\min\{x-a_n,b_n-x\}$, we see that $M_f(x,s_n)\le \frac1{2^{n-2}}$.  Since $s_n \to 0$, it follows that $\lip f(x)=0$ and we are finished with the proof of Theorem  \ref{trim g delta}.
\end{proof}

\section{Proof of Theorem \ref{monotone}}

\begin{proof}

Assume that $E$ is $G_\delta$ with $|E|=0$.  We assume without loss of generality that $E \subset (0,1)$.  To get started we need a few definitions and some helpful lemmas.

\begin{definition}\label{n small}
Suppose that $U$ is open, $a < b$ and $n \in \nn$.  Then we say that $U$ is $n$-small on $(a,b)$ if
$E \subset U$ and
\beq\label{n small def}
|U\cap [a+\frac{b-a}{2^{j+1}},a+\frac{b-a}{2^j}]|=|U\cap[b-\frac{b-a}{2^j},b-\frac{b-a}{2^{j+1}}]|=\frac{b-a}{2^{2n+j+1}} \, \forall j \in \nn.
\eeq
Furthermore, if $V$ is an open set and $U$ is $n$-small on each component of $V$, then we say that $U$ is $n$-small on $V$.
\end{definition}

Note that if $U$ is $n$-small on $(a,b)$, then $|U\cap (a,b)|=2^{-2n}(b-a)$.  Moreover, because $|E|=0$, given any open set $V$ such that $E\subset V$ and $n \in \nn$, we can find an open set $U$ which is $n$-small on $V$.

\begin{definition}\label{def of h}
Suppose that the open set $U$ is $n$-small on $(a,b)$.  Then we define $g=h(n,U,(a,b))$ on $[a,b]$ as follows:
\beq\label{h def}
g(x)=2^n \cdot |U\cap[a,x]|.
\eeq
\end{definition}
The proofs of the following two lemmas are straightforward and left to the reader.

\begin{lemma}\label{Mf est lemma}
Suppose that $U=\cup_{j=1}^\infty (a_j,b_j)$ is $n$-small on $(c,d)$,
$$g=h(n,U,(c,d)),$$
$$f(x)=g(x) \mbox{ for all } x \in [c,d]\backslash U,$$
and
$$g(a_j)\le f(x) \le g(b_j) \mbox{ for all } x \in (a_j,b_j) \mbox{ and for all }j \in \nn.$$
Then for all $x \in [c,d]$ we have
\beq\label{f est}
\max\{0,2^{-n+1}(x-\frac{c+d}2)\} \le f(x) \le \min\{2^{-n+1}(x-c),2^{-n}(d-c)\}
\eeq
and letting $r_x = \min\{x-c,d-x\}$, we have
\beq\label{Mf ineq}
M_f(x,r_x) \le 2^{-n+2} \mbox{ for all }x \in (c,d).
\eeq
\end{lemma}

\begin{lemma}\label{f lipschitz}
Suppose that the open set $V$ is trim in $[a,b]$ and $h$ is linear on $[a,b]$ with slope $m>1$ and $n \in \nn$.  Set $\alpha=|[a,b]\backslash V|$ and define
$f=f(h,[a,b],V,n)$ on $[a,b]$ as follows:
$$f(x)=h(a)+(\frac{(m-2^{-n+1})(b-a)}\alpha+\frac1{2^n})|[a,x]\backslash V|+\frac1{2^n}|[a,x]\cap V|.$$
Then $f$ is increasing on $[a,b]$ and
\beq\label{f=h}
f(a)=h(a) \mbox{ and } f(b)=h(b),
\eeq
\beq\label{f slope 2 -n}
f \mbox{ is linear with slope } 2^{-n} \mbox{ on each component of }(a,b) \cap V,
\eeq
and
\beq\label{f is lipschitz}
f \mbox{ is Lipschitz on }[a,b].
\eeq

\end{lemma}

\begin{definition}\label{trim on U}
Suppose that $U$ and $V$ are open sets and $V$ is trim on each component of $U$.  Then we say that $V$ is trim on $U$.
\end{definition}

Let $V_1=(0,1)$ and define $\{U_n\}$ and $\{V_n\}$ inductively to satisfy the following conditions:

\beq\label{Un small on Vn}
U_n \mbox{ is $n$-small on }V_n \mbox{ for all } n \in \nn
\eeq

\beq\label{Vn trim on Un}
V_{n+1} \mbox{ is trim on }U_n \mbox{ for all } n \in \nn
\eeq

\beq\label{Vn in Un}
V_{n+1} \subset U_n \subset V_n \mbox{ for all }n \in \nn
\eeq

\beq\label{intersect Vn}
\cap_{n=1}^\infty V_n =E.
\eeq

For each $n \in \nn$ we define $\mathcal{U}_n$ to the the collection of all components of $U_n$ and $\mathcal{V}_n$ to be the collection of all
components of $V_n$.

We now begin the construction of a function $f$ satisfying the conclusions of the theorem.   We begin by defining sequences of functions $\{f_n\}$ and $\{g_n\}$ on $[0,1]$ and then defining $f=\lim_{n \to \infty}f_n=\lim_{n\to \infty}g_n$.  The $g_n$s will be constructed to ensure that $\Lip f=\infty$ on $E$ and the $f_n$s will be constructed to ensure that $\lip f=0$ on $E$.   First of all, define $f_1(x)=\frac12x$ on $[0,1]$.  Then construct the $f_n$s and $g_n$s inductively to satisfy the following for all $n \in \nn$:

\beq\label{gn restricted to I}
\mbox{ for all }I=(a,b)\in \mathcal{V}_n, \, g_n\vert_{\overline{I}}=h(n,U_n,I)+f_n(a)
\eeq
and
\beq\label{gn equal fn}
g_n=f_n \mbox{ on }[0,1]\backslash V_n
\eeq
\beq\label{fn+1=f}
\mbox{ for all }J \in \mathcal{U}_n, \, f_{n+1}\vert_{\overline{J}}=f(g_n,\overline{J},V_{n+1},n+1)
\eeq
and
\beq\label{fn+1=gn}
f_{n+1}=g_n \mbox{ on }[0,1]\backslash U_n.
\eeq

The following lemma follows easily from the definitions above.  We leave the proof to the reader.

\begin{lemma}\label{fn gn Lipschitz}
For every $n\in \nn$ we have the following:
\beq\label{fn and gn lipschitz}
f_n \mbox{ and }g_n \mbox{ are Lipschitz on }[0,1],
\eeq
\beq\label{fn gn increasing}
f_n \mbox{ and }g_n \mbox{ are increasing on }[0,1],
\eeq
\beq\label{gn squeeze}
\mbox{ for all }J=(c,d)\in \mathcal{V}_n, \, f_n(c)\le g_n(x) \le f_n(d) \, \forall x \in \overline{J},
\eeq
\beq\label{fn squeeze}
\mbox{ for all }I=(a,b)\in \mathcal{U}_n \ g_n(a) \le f_{n+1}(x) \le g_n(b) \, \forall x \in \overline{I}.
\eeq
\end{lemma}

It follows from (\ref{Vn in Un}) and (\ref{gn restricted to I}) - (\ref{fn+1=gn}) that for all $k \ge n\ge 1$ we have:
\beq\label{fn less than fk}
f_n(c)\le f_k(x)\le f_n(d)=f_n(c)+\frac{d-c}{2^n}\, \forall x \in J=(c,d)\in \mathcal{V}_n 
\eeq
and
\beq\label{fk=fn}
f_k(x)=f_n(x) \, \forall x\in [0,1]\backslash V_n.
\eeq
Therefore, we may conclude that
\beq\label{norm ineq}
\vert\vert f_k - f_j \vert\vert_\infty \le \max_{J \in \mathcal{V}_n} \frac{|J|}{2^n} \ \forall j,k \ge n.
\eeq
Since $\max_{J \in \mathcal{V}_n} \frac{|J|}{2^n} \to 0$ as $n \to \infty$, it follows that $\{f_n\}$ converges uniformly to
a function $f$ on $[0,1]$.  Moreover, using (\ref{gn equal fn}) and (\ref{gn squeeze}), we may conclude that $\{g_n\}$ converges uniformly to $f$ as well.
Since each $g_n$ is continuous on $[0,1]$, it follows that $f$ is also continuous on $[0,1]$.  We also note that (\ref{Vn in Un}) and (\ref{gn equal fn}) - (\ref{fn+1=gn})  imply that for all $n \in \nn$ and for all $k \ge n$, we have
$$g_n(a) \le g_k(x) \le g_n(b) \mbox{ for all }x \in (a,b) \in \mathcal{U}_n,$$
and therefore for $n \in \nn$ we have
\beq\label{f squeeze gn}
g_n(a)\le f(x) \le g_n(b) \mbox{ for all } x \in (a,b) \in \mathcal{U}_n.
\eeq
Now note that since $f =\lim f_n$, (\ref{fk=fn}) implies that
\beq\label{f equal fn}
f(x)=f_n(x) \mbox{ for all }x \in [0,1]\backslash V_n.
\eeq
Furthermore, using (\ref{Vn in Un}), (\ref{fn+1=gn}) and (\ref{f equal fn}), we get
\beq\label{f equal gn}
f(x)=g_n(x) \mbox{ for all }x \in [0,1]\backslash U_n.
\eeq
Let $n \in \nn$ and $J =(c,d)\in \mathcal{V}_n$.  Then $U_n$ is $n$-small on $J$,
$$g_n\vert_J=h(n,U_n,J)+f_n(c)=h(n,U_n,J)+f(c),
$$
and $f(x)=g_n(x)$ for all $x \in J \backslash U_n$ by (\ref{f equal gn}).  Using (\ref{f squeeze gn}), we can apply Lemma \ref{Mf est lemma} to conclude that (\ref{f est}) holds
for all $x \in J$.

From (\ref{f est}), (\ref{f equal fn}) and (\ref{fn and gn lipschitz}) we deduce that $\Lip f$ is bounded on $[0,1]\backslash V_n$ and therefore by (\ref{intersect Vn})
we get that $\Lip f < \infty$ on $[0,1]\backslash E.$

It remains to show that $\Lip f=\infty$ on $E$ and $\lip f =0$ on $E$.  Let $x\in E$.

We first show that $\Lip f(x) = \infty$.  For each $n \in \nn$ choose $(a_n,b_n)\in \mathcal{U}_n$ such that $x \in (a_n,b_n)$.  From (\ref{gn restricted to I}) and (\ref{h def}) we see that $g_n(b_n)-g_n(a_n)=2^n(b_n-a_n)$.  However, by (\ref{f equal gn}) we have that $f(a_n)=g_n(a_n)$ and $f(b_n)=g_n(b_n)$ and therefore $M_f(x,r_n) \ge 2^{n-1}$ where
$r_n$ is either equal to $x-a_n$ or $b_n-x$.  Letting $n \to \infty$, we conclude that $\Lip f(x)=\infty$.

Finally, we show that $\lip f(x)=0$.  For each $n \in \nn$ choose $(c_n,d_n)\in \mathcal{V}_n$ such that $x \in (c_n,d_n)$ and let $r_n =\min\{x-c_n,d_n-x\}$.  Then from
(\ref{Mf ineq}) we conclude that $M_f(x,r_n)\le 2^{-n+2}$.  Letting $n \to \infty$ we deduce that $\lip f(x) =0$ and we are done with the proof of Theorem \ref{monotone}.

\end{proof}

\section{Some Related Problems}

Let $\l_f^0=\{x \in \rr \ : \ \mbox{lip}f(x)=0\}$.
There are a number of open problems concerning the relationship between the sets $L_f^\infty$, $l_f^\infty$ and $l^0_f$. 
For example, it would be interesting to characterize the sets $E \subset \rr$ that satisfy any of the following conditions.   
\be

\item $E\subset L^\infty_f \cap l^0_f$.
\medskip 

\item $E\subset L^\infty_f \cap l^0_f$ and $l_f^\infty =\emptyset$.

\item $E=L^\infty_f \cap l^0_f$. 

\item   $E=L^\infty_f \cap l^0_f$ and $l_f^\infty = \emptyset$. 

\item $L^\infty_f=\rr$ and $\l^0_f=E$.   

\item $L^\infty_f=l^0_f=E$.

\ee


\end{document}